\documentclass[12pt,a4paper]{amsart}

\usepackage{color,amsmath,amssymb,latexsym,amsfonts,setspace}

\normalsize \evensidemargin=-12pt \oddsidemargin=-12pt
\newtheorem{theorem}{Theorem}
\newtheorem{corollary}[theorem]{Corollary}

\newtheorem{lemma}[theorem]{Lemma}
\newtheorem{proposition}[theorem]{Proposition}
\newtheorem{remark}[theorem]{Remark}
\def\bds{\begin{displaystyle}}
\def\eds{\end{displaystyle}}

\begin{document}

\title{The mixing advantage is less than 2}
\author{K. Hamza}
\address{Kais Hamza, School of Mathematical Sciences, Monash University}
\author{P. Jagers}
\address{Peter Jagers, Mathematical Statistics, Chalmers University of Technology}
\author{A. Sudbury}
\address{Aidan Sudbury, School of Mathematical Sciences, Monash University}
\author{D. Tokarev}
\address{Daniel Tokarev, School of Mathematical Sciences, Monash University}

\begin{abstract}
Corresponding to $n$ independent non-negative random variables
$X_1,\ldots,X_n$, are values $M_1,\ldots,M_n$, where each $M_i$ is
the expected value of the maximum of $n$ independent copies of
$X_i$. We obtain an upper bound to the expected value of the
maximum of $X_1,\ldots,X_n$ in terms of $M_1,\ldots,M_n$. This
inequality is sharp in the sense that the quantity and its bound
can be made as close to each other as we want. We also present
related comparison results.
\end{abstract}


\maketitle

\markboth{{\normalsize\sc Hamza, Jagers, Sudbury \&
Tokarev}}{{\normalsize\sc The mixing advantage is less than 2}}

\section{Introduction}

To illustrate the main thrust of this paper, we consider a simple
two-component parallel system; say a light pole made up of two
light bulbs. If the system is considered to have failed once both
components, assumed to act independently, fail, then a reasonable
measure of the performance of the system is
$\mathbb{E}[\max(X,Y)]$, where $X$ and $Y$ are the independent
random lifetimes of the two components.

Assume that, to build the system, we may choose from any of two
manufacturers (i.e. two lifetime distributions). Should we choose
two components from the same manufacturer or should we mix? In the
case of manufacturers with identical performances
($\mathbb{E}[\max(X_1,X_2)] = \mathbb{E}[\max(Y_1,Y_2)]$), there
is (almost) always a net gain in mixing. It is then natural to ask
how much gain can one achieve and further to identify situations
in which this gain is attained.

In the spirit of the above example, we shall call a family of $n$
independent random variables an $n$-assembly. When these are also
identically distributed, we will say that they form a similar
$n$-assembly. We call performance of an $n$-assembly (whether
similar or not) the expected value of its maximum. The aim of this
paper is to bound (from above as well as below) the performance of
an $n$-assembly relying solely on the performances of all similar
$n$-assemblies from which it is drawn. It will be shown that
mixing (i.e. using assemblies issued from different distributions)
improves performance by a factor, hereby called the mixing factor,
of up to (but not including) 2. We will further show that when all
similar $n$-assemblies have the same performance, the mixing
factor is at least 1.

While an extensive literature exists on the expected value of the
maximum of $n$ independent and identically distributed random
variables (see for example \cite{DavNag03}), with the exemption of
\cite{ArnGro79} and \cite{Sen70} (see also \cite{BalBal08} and
\cite[Section 5.2]{DavNag03}), not much work has concentrated on
the case of non-identically distributed random variables.
Furthermore, the aforementioned papers do not attempt a comparison
with $M_1,\ldots,M_n$. In \cite{ArnGro79}, the authors obtain
upper and lower bounds in terms of $\mathbb{E}[X_i]$ and
$\mathrm{var}(X_i)$ (assumed to be finite) where $X_1,\ldots,X_n$
are possibly dependent random variables with possibly different
distributions. These bounds generalise those of \cite{HarDav54}
and \cite{Gum54} which deal with the independent and identically
distributed case. In \cite{Sen70}, the author solely focuses on
obtaining a lower bound. This is done by comparing the
distribution function of the maximum of $X_1,\ldots,X_n$ to that
of the maximum of $n$ independent copies of an equally-weighted
probability mixture of $X_1,\ldots,X_n$ (see later). More
recently, in \cite{Tok07}, the author investigates the performance
of an $n$-assembly constructed from 2 distributions and its
behaviour as the make-up of the $n$-assembly changes.

\subsection*{Notations and Assumptions}
Throughout this paper, we use the following notations and
assumptions. For any sequence $x_1,\ldots,x_n$, we write $x_1\vee
x_2\vee\ldots\vee x_n = \bigvee_{k=1}^nx_k$ for the maximum
$\max(x_1,x_2,\ldots,x_n)$. Random variables will generally be
indexed in the following way: $X_i^j$ refers to the $j$th element
in the similar $n$-assembly $i$; i.e. $X_i^1,\ldots,X_i^n$ are
independent random variables having the same distribution as
$X_i$. $X_i^{(n)} = \bigvee_{k=1}^nX_i^k$, $X_{(n)} =
\bigvee_{i=1}^nX_i$, $M_i = \mathbb{E}[X_i^{(n)}]$, $M_{(n)} =
\bigvee_{i=1}^nM_i$ and $\bar M = \frac1n\sum_{i=1}^nM_i$.
Finally, we denote by $\theta_n$ the mixing factor of a set of
$n$-assemblies:
$$\theta_n = \frac{\mathbb{E}[X_{(n)}]}{M_{(n)}}.$$
We assume that all random variables are non-negative and have
finite first moment.

\section{The Mixing Factor Bounds}
\subsection{The Main Results}
The main theorem which we wish to show is:
\begin{theorem}
Let $X_1,\ldots,X_n$ be independent random variables and
$M_i=\mathbb{E}[X_i^{(n)}]$, $i=1,2,...,n$. Then
\begin{equation}
\bar M\leq\mathbb{E}[X_{(n)}]\leq\bar
M+\frac{n-1}{n}M_{(n)}.\label{MainInq}
\end{equation}

In particular, if $M_i=M$, $i=1,...,n$,
$$M\leq\mathbb{E}[X_{(n)}]\leq(2-1/n)M.$$
\label{MainTh}
\end{theorem}

\begin{corollary}
For $n$-assemblies, the mixing factor does not exceed $2-1/n$ and,
in the case of equally performing similar $n$-assemblies, it is at
least 1.
\end{corollary}

In the case where some of the random variables $X_1,\ldots,X_n$
are identically distributed, an improved upper bound may be
achieved. Such a bound follows immediately from Theorem
\ref{MainTh} when all distributions are repeated an equal amount.

\begin{corollary}
Suppose the $n$-assembly $X_1,\ldots,X_n$ is made up of $k$
similar $m$-assemblies ($n=km$). Then
$$\bar M\leq\mathbb{E}[X_{(n)}]\leq\bar M+\frac{k-1}{k}M_{(n)}.$$
\end{corollary}


\begin{remark}
Let $F_i$ denote the distribution function of $X_i$ and
$G_i=F_i^n$ be the distribution function of $X_i^{(n)}$. Then
(\ref{MainInq}) can be rewritten as
\begin{equation}
0\leq\int_0^\infty(\bar G(s)-\tilde G(s))ds\leq(1-1/n)\max_{1\leq
i\leq n}\int_0^\infty(1-G_i(s))ds,\label{GAM}
\end{equation}
where $\bar G$ and $\tilde G$ are the arithmetic and geometric
means, respectively, of the distribution functions
$G_1,G_2,\ldots,G_n$.

Equation (\ref{GAM}) provides an upper bound (the lower bound is a
direct consequence of the arithmetic and geometric mean
inequality) on the $L^1$-distance between the geometric and
arithmetic means of a sequence of $n$ distribution functions on
the positive half-line with finite means.

In probabilistic terms, this expresses the following fact. Let
$Y_1,\ldots,Y_n$ be independent random variables (non-negative
with finite mean). Let $U$ be an equally-weighted probability
mixture of $Y_1,\ldots,Y_n$ and $V$ be such that the distribution
of the maximum of $n$ independent copies of $V$ is that of
$Y_{(n)}$. Then
$$0\leq\mathbb{E}[V]-\mathbb{E}[U]\leq(1-1/n)\max_{1\leq
i\leq n}\mathbb{E}[Y_i].$$
\end{remark}

As hinted in the introduction, the lower bound is not optimal. In
\cite{Sen70} (see also \cite{BalBal08} and \cite[Section
5.2]{DavNag03}), it is shown that
\begin{equation}
\mathbb{P}(X_{(n)}\leq x) \leq \mathbb{P}(Z^{(n)}\leq
x),\quad\forall x>0,\label{Sen}
\end{equation}
where $Z^1,\ldots,Z^n$ are independent copies of an
equally-weighted probability mixture of $X_1,\ldots,X_n$.
Combining the inequality of the arithmetic and geometric means
with a moment inequality, we can further bound the right hand side
of (\ref{Sen})
$$\mathbb{P}(Z^{(n)}\leq x) =
\left(\frac1n\sum_{k=1}^n\mathbb{P}(X_k\leq x)\right)^n \leq
\frac1n\sum_{k=1}^n\mathbb{P}(X_k\leq x)^n =
\frac1n\sum_{k=1}^n\mathbb{P}(X_k^{(n)}\leq x).$$ As a
consequence, we obtain an improved lower bound for
$\mathbb{E}[X_{(n)}]$,
$$\bar M\leq\mathbb{E}[Z^{(n)}]\leq\mathbb{E}[X_{(n)}]$$
Although a better bound for the performance of an $n$-assembly,
$\mathbb{E}[Z^{(n)}]$ cannot be expressed in terms of the
performances of all similar $n$-assemblies from which it is drawn.
In other words, $\mathbb{E}[Z^{(n)}]$ is not a bound that can be
expressed in terms of $M_1,\ldots,M_n$. Furthermore, it is shown
in Corollary \ref{BoundCase} that, in the case of bounded random
variables, an improved lower bound, expressed in terms of
$M_1,\ldots,M_n$, can be achieved. This lower bound is shown, in
the special case of a two-point distribution, to outperform
$\mathbb{E}[Z^{(n)}]$.

\subsection{Two Toy Examples}
Before we embark on the proof of the main Theorem, we look at the
simple case of two-point distributions. The first example
presented here will give some insight into how the improved bound
of Corollary \ref{BoundCase} and that of \cite{Sen70} compare. The
second example will demonstrate the sharpness of the upper bound
of Theorem \ref{MainTh}; i.e. we show that $\mathbb{E}[X_{(n)}]$
can be made as close to the upper bound as we want.

\begin{enumerate}

\item Assume that the random variables $X_1,\ldots,X_n$ are all
concentrated on two points, 0 and $b$, and let $p_i =
\mathbb{P}[X_i=0]$. Then the distribution of $Z$, the
equally-weighted probability mixture of $X_1,\ldots,X_n$, is given
by $\bds\mathbb{P}[Z=0] = \frac1n\sum_{i=1}^np_i = \bar p\eds$
and,
$$\mathbb{E}[Z^{(n)}] = b(1-\bar p^n)\mbox{ and }M_i = b(1-p_i^n).$$
Therefore,
$$b-\prod_{i=1}^n\sqrt[n]{b-M_i}-\mathbb{E}[Z^{(n)}] =
b\bar p^n-b\prod_{i=1}^np_i \geq 0,$$ by a simple application of
the inequality for the arithmetic and geometric means.

\item Again, we assume that the random variables $X_1,\ldots,X_n$
are all concentrated on two points, or less. However, in this
case, we allow the non-zero values $x_1,\ldots,x_n$, to be
different. In fact, we assume (without loss of generality) that
$X_n$ is non-random, that all other random variables take 0, with
probability $\bds p_k = \sqrt[n]{1-M_k/x_k} > 0\eds$, and $x_k$,
with probability  $1-p_k > 0$, and that
$$M_{(n-1)} \leq M_n < x_1 < \ldots < x_{n-1}.$$
Then, it is shown in the proof of Proposition \ref{TwoPointProp}
(from which this very construction is extracted), that
$$\mathbb{E}[X_{(n-1)}\vee M_n] = p_{n-1}\mathbb{E}[X_{(n-2)}\vee M_n] +
(1-p_{n-1})\frac{M_{n-1}}{1-p_{n-1}^n}.$$ Letting $p_{n-1}$ go to
1 and $x_{n-1}$ to infinity so that $M_{n-1}$ remains constant, we
see that $\mathbb{E}[X_{(n)}]$ approaches
$\bds\mathbb{E}[X_{(n-2)}\vee M_n] + \frac{M_{n-1}}{n}\eds$. Now
letting $p_{n-2}$ go to 1 and $x_{n-2}$ to infinity so that
$M_{n-2}$ remains constant, we see that $\mathbb{E}[X_{(n)}]$
approaches $\bds\mathbb{E}[X_{(n-3)}\vee M_n] +
\frac{M_{n-2}+M_{n-1}}{n}\eds$. Repeating this process leads to
the fact that $\mathbb{E}[X_{(n)}]$ approaches $\bds
M_n+\frac{M_1+\ldots+M_{n-1}}{n} = \bar M +
\frac{n-1}nM_{(n)}\eds$ (recall that $M_n = M_{(n)}$).

\end{enumerate}

\subsection{The Proofs}
The lower bound is a direct consequence of the inequality for
arithmetic and geometric means and is given here for completeness
only.

\begin{proposition}
If $X_1,\ldots,X_n$ are independent random variables with the
property that $\mathbb{E}[X_i^{(n)}]=M_i$, $i=1,2,...,n$ , then
$\mathbb{E}[X_{(n)}]\geq\bar M$. Furthermore, if
$\mathbb{E}[X_{(n)}] = \bar M$ then $X_1,\ldots,X_n$ are
identically distributed.\label{LwrBnd}
\end{proposition}
\begin{proof}
Because the arithmetic mean is greater than the geometric mean,
\begin{equation}
1-u_1u_2...u_n\geq\frac{1}{n}\sum_{i=1}^n(1-u_i^n).\label{GAInq}
\end{equation}
This implies
$$\mathbb{E}[X_{(n)}] = \int_0^{\infty}[1-F_1(x)F_2(x)...F_n(x)]dx \geq \frac{1}{n} \sum_{i=1}^n
 \int_0^{\infty}[1-F_i^n(x)]dx = \bar M,$$
where $F_i$ denotes the distribution function of $X_i$. Further,
since (\ref{GAInq}) turns into an equality if and only if the
$u_i$'s are all equal, the lower bound of $\mathbb{E}[X_{(n)}]$ is
only attained when $F_1=\ldots=F_n$, that is when $X_1,\ldots,X_n$
are identically distributed.
\end{proof}

The first step in the proof of the upper bound is to reduce the
problem to the case of random variables concentrated on a finite
set of points. This is easily demonstrated by using the
approximation
$$X = \lim_{m\uparrow\infty}\left[
\sum_{l=1}^{m2^m}\frac{l-1}{2^m}1_{[(l-1)/2^m,l/2^m)}(X) +
m1_{[m,+\infty)}(X)\right].$$ Indeed, one only needs to apply the
Monotone Convergence Theorem to prove the following proposition.

\begin{proposition}
If Theorem \ref{MainTh} is true for random variables concentrated
on a finite set of points, then it must be true for general random
variables.\label{FiniteTh}
\end{proposition}

Having reduced the problem to one that only involves random
variables concentrated on a finite set of points, we shall find
the maximum possible value of $\mathbb{E}[X_{(n)}]$ when we fix
$\mathbb{E}[X_i^{(n)}]=M_i$, $i=1,2,...,n$, by continually
creating new sets of $n$-assemblies which maintain the property
$\mathbb{E}[X_i^{(n)}]=M_i$, but increase, or at least do not
decrease, $\mathbb{E}[X_{(n)}]$.

The main step in achieving the upper bound announced in Theorem
\ref{MainTh} is to prove that the problem can be reduced to one
that involves random variables that take exactly one non-zero
value. This is done by showing that, for any random variable, two
non-zero adjacent points can be coalesced into a single point;
that is, a random variable which takes values $x_1,\ldots,x_r$
($x_1<\ldots<x_r$) with probabilities $p_1,\ldots,p_r$
respectively, can be replaced by a random variable with masses
$p_1,\ldots,p_{i-1},p_i+p_{i+1},p_{i+2},\ldots,p_r$ at
$x_1,\ldots,x_{i-1},x,x_{i+2},\ldots,x_r$, for a carefully chosen
$x$. This is initially done in Proposition \ref{coalesce} for
adjacent points not separated by points from other random
variables, and later extended to the general case in Proposition
\ref{reduce} yielding the following Theorem.

\begin{proposition}
If Theorem \ref{MainTh} is true for random variables which take
exactly one non-zero value, then it must be true for general
random variables.\label{TwoPointTh}
\end{proposition}

\begin{proposition}
Let $X_1,\ldots,X_n$ be independent random variables with the
property that $\mathbb{E}[X_i^{(n)}]=M_i$, $i=1,2,...,n$. If each
of $X_1,\ldots,X_n$ takes exactly one non-zero value, then
$$\bar M\leq\mathbb{E}[X_{(n)}]\leq\bar
M+\frac{n-1}{n}M_{(n)}.$$\label{TwoPointProp}
\end{proposition}
\begin{proof}
For $k\in\{1,\ldots,n\}$, we denote by $x_k$ the non-zero value of
$X_k$ and $p_k = \mathbb{P}[X_k=0]$ . The next four points
successively simplify the problem.
\begin{enumerate}

\item If $\mathbb{P}[X_k=0]>0$, for all $k$ (i.e. all $X_k$'s
place a positive mass at 0), then applying Proposition
\ref{coalesce} to $X_1$ (assumed without loss of generality to
have the smallest non-zero value) and $Y = \bigvee_{k=2}^nX_k$
shows that $\mathbb{E}[X_{(n)}]$ is increased if $X_1$ is replaced
by $M_1 = \mathbb{E}[X_1^{(n)}]$. Therefore we may assume that at
least one of the $X_k$'s is constant.

\item Without loss of generality, we may assume that $X_n$ is the
largest constant random variable. $X_n$ may be one of any number
of constant random variables, all equal, that are larger than any
other constant random variable. Assume that at least one constant
random variable is strictly smaller than $X_n$. Let it be $a$
($a<X_n$). Since $\mathbb{E}[a\vee X_n\vee Y] = \mathbb{E}[Z\vee
X_n\vee Y]$, where $Z$ is such that
$\mathbb{P}[Z=0]=(1-a/X_n)^{1/n}$ and
$\mathbb{P}[Z=X_n]=1-(1-a/X_n)^{1/n}$ ($a\vee X_n=X\vee X_n=X_n$),
we may assume that all $X_k$'s satisfying $M_k<M_n$, are
non-constant (place a positive mass on 0).

\item If one other $X_k$, say $X_{n-1}$, satisfies $X_{n-1} =
M_{n-1} = M_n$, then replacing $X_{n-1}$ by $Y$ such that
$\mathbb{P}[Y=0]=q>0$, $\mathbb{P}[Y=y]=1-q$ and
$y=M_n/(1-q^n)>M_n$, we get
\begin{eqnarray*}
\lefteqn{\mathbb{E}[X_{(n-2)}\vee Y\vee M_n]}\\
& = & q\mathbb{E}[X_{(n-2)}\vee M_n] + (1-q)\mathbb{E}[X_{(n-2)}\vee y\vee M_n]\\
& \geq & \mathbb{E}[X_{(n-2)}\vee M_n] = \mathbb{E}[X_{(n)}].
\end{eqnarray*}
Therefore we may assume that the random variable with the largest
$M_k$ (say $M_n$) is constant and all others are non-constant.

\item Without loss of generality, we may assume that
$x_1\leq\ldots\leq x_{n-1}$. Recall that, for
$k\in\{1,\ldots,n-1\}$, $M_k = (1-p_k^n)x_k$. If $x_1\leq M_n$
then, with $Y$ such that $\mathbb{P}[Y=0]=q>0$,
$\mathbb{P}[Y=y]=1-q$ and $y=M_1/(1-q^n)>M_n$,
$$\mathbb{E}[Y\vee Z\vee M_n]
= q\mathbb{E}[Z\vee M_n] + (1-q)\mathbb{E}[y\vee Z\vee M_n] >
\mathbb{E}[X_{(n)}],$$ where $Z=\bigvee_{k=2}^{n-1}X_k$. Therefore
we may assume that $M_{(n-1)}\leq M_n<x_1\leq\ldots\leq x_{n-1}$.

\end{enumerate}
Now,
\begin{eqnarray*}
\mathbb{E}[X_{(n)}] & = & \mathbb{E}[X_{(n-2)}\vee X_{n-1}\vee
M_n]\\
& = & p_{n-1}\mathbb{E}[X_{(n-2)}\vee M_n] + (1-p_{n-1})x_{n-1}\\
& = & p_{n-1}\mathbb{E}[X_{(n-2)}\vee M_n] +
(1-p_{n-1})\frac{M_{n-1}}{1-p_{n-1}^n}
\end{eqnarray*}
Next we use the fact that the function $\phi(p) =
up+v\frac{1-p}{1-p^n}$ defined on $[0,1]$ (extended at 1 by
continuity), where $v<u$, increases from $v$ to $u+v/n$. Since
$\mathbb{E}[X_{(n-2)}\vee M_n]>M_{n-1}$, we immediately get that
$$\mathbb{E}[X_{(n)}]\leq\mathbb{E}[X_{(n-2)}\vee
M_n]+\frac{M_{n-1}}{n}.$$ Repeating the same argument, we get
$$\mathbb{E}[X_{(n)}]\leq\mathbb{E}[X_{(n-3)}\vee
M_n]+\frac{M_{n-2}+M_{n-1}}{n}\leq\ldots\leq
M_n+\frac{M_1+\ldots+M_{n-1}}{n}.$$
\end{proof}

\section{Comparison Results}

In this section we develop the tools required to obtain the bounds
of the previous section. However, these tools are important in
their own right. They enable us to increase the performance of an
$n$-assembly while keeping the performances of all similar
$n$-assemblies unchanged. We prove them for random variables that
are not necessarily concentrated on a finite set of points. We
also obtain a comparison result that allows to decrease the
performance of an $n$-assembly while keeping the performances of
all similar $n$-assemblies unchanged.

\begin{proposition}
If $X_1$ and $Y$ are independent, $\mathbb{P}[a\leq X_1\leq b]=p$
and $\mathbb{P}[a<Y<b]=0$ then we may replace $X_1$ by a random
variable $X_2$ such that $X_2=X_1$ outside $[a,b]$ and,
$\mathbb{P}[a\leq X_2\leq b]=\mathbb{P}[X_2=x]=p$, for some
$a<x<b$ with the property that
$\mathbb{E}[X_2^{(n)}]=\mathbb{E}[X_1^{(n)}]$ and
$\mathbb{E}[X_2\vee Y]\geq\mathbb{E}[X_1\vee Y]$.\label{coalesce}
\end{proposition}
\begin{proof}
Equating the contributions to $\mathbb{E}[X_1^{(n)}]$ and
$\mathbb{E}[X_2^{(n)}]$ from the interval $[a,b]$ gives
$$x(F(b)^n-F^-(a)^n) = \sum_{k=1}^n{n\choose k}p^kF^-(a)^{n-k}
\mathbb{E}[X_1^{(k)}|a\leq X_1^j\leq b,j=1,\ldots,k]$$ where $F$
is the distribution function of $X_1$ and
$F^-(x)=\lim_{y\uparrow x}F(y)$. We thus see that $x$ is a convex
combination of the $n$ expectations on the right hand side and
thus, in particular, that $x\geq\min_{1\leq k\leq
n}\mathbb{E}[X_1^{(k)}|a\leq X_1\leq b]=\mathbb{E}[X_1|a\leq
X_1\leq b]$. Finally
$$\mathbb{E}[X_2\vee Y] - \mathbb{E}[X_1\vee Y] =
p\mathbb{P}[Y<a](x-\mathbb{E}[X_1|a\leq X_1\leq b])$$ and, as
previously observed, the second expression is nonnegative.
\end{proof}

If we allow $\mathbb{P}[a<Y<b]>0$ then the coalescing point is no
longer necessarily in the interval $[a,b]$ as is demonstrated in
the next result.

\begin{proposition}
Assume that $X_1$ and $Y$ are independent and that the entire mass
$X_1$ places on an interval $(l,r)$ is concentrated on two values
$a$ and $b$ within it: $0\leq l<a<b<r$,
$\mathbb{P}[l<X_1<r]=\mathbb{P}[X_1\in\{a,b\}]$,
$p=\mathbb{P}[X_1=a]>0$ and $q=\mathbb{P}[X_1=b]>0$. Then there
exists a random variable $X_2$ s.t.
$\mathbb{E}[X_1^{(n)}]=\mathbb{E}[X_2^{(n)}]$, $\mathbb{E}[X_1\vee
Y]<\mathbb{E}[X_2\vee Y]$ and the mass $X_2$ places on the
interval $(l,r)$ is concentrated on at most one single value
within it.\label{reduce}
\end{proposition}
\begin{proof}
Let $X_2 = X_11_{X_1\leq l} + u1_{X_1=a} + V(u)1_{X_1=b} +
X_11_{X_1\geq r}$, where $u$ is a free parameter,
$$V(u) = b-\lambda_n(u-a)\mbox{ and }\lambda_n = \frac{F(a)^n-F(l)^n}{F(b)^n-F(a)^n}.$$
Then $\mathbb{E}[X_2^{(n)}]=\mathbb{E}[X_1^{(n)}]$ and
$\mathbb{E}[X_2\vee Y] = \mathbb{E}[X_1\vee Y] + \phi(u) -
\phi(a)$, where
$$\phi(u) = p\mathbb{E}[u\vee Y] + q\mathbb{E}[V(u)\vee Y].$$
Using Lemma \ref{convex} below, we get that $\phi^\prime_+(u) =
pG(u) - q\lambda_nG_-(V(u))$ is non-decreasing and that
$\phi^\prime_+(a) = pG(a) - q\lambda_nG_-(b)$ is not zero unless
$G(a)>0$ (since $G_-(b)\geq\mathbb{P}[a<Y<b]>0$).

Therefore there are three possible scenarios, either
$\phi^\prime_+(a)<0$, $\phi^\prime_+(a)>0$ or,
$\phi^\prime_+(a)=0$ and $G(a)>0$. In the first case,
$\mathbb{E}[X_2\vee Y]$ decreases on $[l,a]$ and letting $u=l$
achieves the objective. In the second case, $\mathbb{E}[X_2\vee
Y]$ increases on $[a,r]$ and it suffices to let $u=r$.

Finally, suppose $\phi^\prime_+(a)=0$ and $G(a)>0$. Let $\omega =
\mathbb{E}[X_1^{(n)}|l<X_1^{(n)}<r]$. Then $a<\omega<b$,
$V(\omega)=\omega$ and
$$\phi^\prime_+(\omega) \geq
G(\omega)\frac{pq}{F(b)^n-F(a)^n}\sum_{k=0}^{n-1}F(a)^k\left[F(b)^{n-k-1}
- F(l)^{n-k-1}\right] > 0.$$ Here it suffices to let $u=\omega$
($a$ and $b$ are merged into $\omega$).

\end{proof}

\begin{lemma}
For any positive random variable $Y$ with distribution function
$G$ and differentiable function $h$, $\gamma(u) =
\mathbb{E}[h(u)\vee Y]$ admits left and right derivatives (that
differ on a set that is at most countable):
$$\gamma^\prime_\pm(u) = \left\{\begin{array}{lcl}
h^\prime(u)G_\pm(h(u))&\mbox{ if }&h^\prime(u)\geq0\\
h^\prime(u)G_\mp(h(u))&\mbox{ if }&h^\prime(u)<0\\
\end{array}\right.$$
\label{convex}
\end{lemma}

The next proposition enables the reduction of
$\mathbb{E}[X_{(n)}]$ for given $M_i$, $i=1,\ldots,n$.
\begin{proposition}
Let $b$ be such that $\mathbb{P}[X_i\leq b]>0$, for all $i$. Then,
for any interval $I=[a,b]$, there are random variables
$Y_1,\ldots,Y_n$ such that
$\mathbb{E}[X_{(n)}]\geq\mathbb{E}[Y_{(n)}]$ and, for all
$i=1,\ldots,n$, $Y_i=X_i$ on $\{X_i\not\in I\}$,
$\mathbb{P}[a<Y_i<b]=0$ and
$\mathbb{E}[X_i^{(n)}]=\mathbb{E}[Y_i^{(n)}]$.\label{down}
\end{proposition}
\begin{proof}
Let $\xi_i$, $i=1,\ldots,n$, be such that
$\mathbb{P}[\xi_i=a]=\alpha_i$, $\mathbb{P}[\xi_i=b]=1-\alpha_i$
and $\xi_1,\ldots,\xi_n$ are independent of each other and of all
other random variables $X_i$. Now let
$$Y_i = \xi_i1_{X_i\in I} + X_i1_{X_i\not\in I}$$
and form the corresponding similar $n$-assemblies,
$(Y_i^1,\ldots,Y_i^n)_{i=1,\ldots,n}$. Our first objective is to
select $\alpha_i$ such that
$\mathbb{E}[X_i^{(n)}]=\mathbb{E}[Y_i^{(n)}]$. For simplicity, we
shall momentarily drop the index $i$ and compute more generally
$\mathbb{E}[X_{(n)}]-\mathbb{E}[Y_{(n)}]$. Further, we introduce
the notation, $\hat z = z\vee a-a = (z-a)^+$ and observe that,
\begin{enumerate}
\item $a\leq z\leq b$ iff $z\geq a$ and $\hat z\leq b-a$
\item $\widehat{z_{(n)}} = \hat z_{(n)}$
\item if $\mathcal{K}$ is set of non-empty subsets of $N=\{1,\ldots,n\}$,
\begin{equation}
\hat z_{(n)}\prod_{k=1}^n1_{z_k\geq a} = \hat z_{(n)} -
\sum_{K\in{\mathcal{K}}}\left(\bigvee_{k\in K}\hat z_k\right)\left(\prod_{k\in
K}1_{z_k\in I}\right)\left(\prod_{k\not\in K}1_{z_k<a}\right).
\label{chapeau}
\end{equation}
\end{enumerate}
For $K\in{\mathcal{K}}$, let
$$A_K = \bigcap_{k\in K}\{X_k\in I\},\
B_K = \bigcap_{k\not\in K}\{X_k<a\},\
C_K = A_K\cap B_K$$
and for any sequence $z_1,\ldots,z_n$,
$z_{[K]} = \bigvee_{k\in K}z_k$.
Using the identities,
$$\{X_{(n)}\in I\} = \bigcup_{K\in{\mathcal{K}}}C_K\mbox{ and }
X_{(n)}-Y_{(n)} =
\sum_{K\in{\mathcal{K}}}(X_{[K]}-\xi_{[K]})1_{C_K},$$ we obtain
\begin{eqnarray*}
\lefteqn{\mathbb{E}[X_{(n)}]-\mathbb{E}[Y_{(n)}]}\\
& = & \sum_{K\in{\mathcal{K}}}\left(\mathbb{E}[X_{[K]},C_K]-\mathbb{E}[\xi_{[K]},C_K]\right)\\
& = & \sum_{K\in{\mathcal{K}}}\left(\mathbb{E}[\hat X_{[K]}+a,C_K]-\mathbb{E}[\hat\xi_{[K]}+a,C_K]\right)\\
& = & \sum_{K\in{\mathcal{K}}}\left(\mathbb{E}[\hat X_{[K]},C_K]-\mathbb{E}[\hat\xi_{[K]},C_K]\right)\\
& = & \sum_{K\in{\mathcal{K}}}\mathbb{P}(B_K)\left(\mathbb{E}[\hat X_{[K]},A_K]-\mathbb{E}[\hat\xi_{[K]},A_K]\right)\\
& = & \mathbb{E}[\hat X_{(n)},A_N] +
\sum_{K\in{\mathcal{K}}\setminus\{N\}}\mathbb{P}(B_K)\mathbb{E}[\hat X_{[K]},A_K]
- \sum_{K\in{\mathcal{K}}}\mathbb{P}(A_K)\mathbb{P}(B_K)\mathbb{E}[\hat\xi_{[K]}]\\
& = & \mathbb{E}[\hat X_{(n)},D_N] - \sum_{K\in{\mathcal{K}}\setminus\{N\}}\mathbb{E}[\hat X_{[K]},D_N\cap B_K\cap E_K]\\
& & + \sum_{K\in{\mathcal{K}}\setminus\{N\}}\mathbb{P}(B_K)\mathbb{E}[\hat X_{[K]},A_K]
- \sum_{K\in{\mathcal{K}}}\mathbb{P}(A_K)\mathbb{P}(B_K)\mathbb{E}[\hat\xi_{[K]}]
\end{eqnarray*}
where $D_K = \bigcap_{k\in K}\{\hat X_k\leq b-a\}$, $E_K = \bigcap_{k\in K}\{X_k\geq a\}$ and
we have used (\ref{chapeau}). Applying the identities $D_N\cap B_K = D_K\cap B_K$ and $D_K\cap E_K = A_K$,
it follows that
\begin{eqnarray}
\lefteqn{\mathbb{E}[X_{(n)}]-\mathbb{E}[Y_{(n)}]}\nonumber\\
& = & \mathbb{E}[\hat X_{(n)},D_N] - \sum_{K\in{\mathcal{K}}\setminus\{N\}}\mathbb{P}(B_K)\mathbb{E}[\hat X_{[K]},A_K]\nonumber\\
& & + \sum_{K\in{\mathcal{K}}\setminus\{N\}}\mathbb{P}(B_K)\mathbb{E}[\hat X_{[K]},A_K]
- \sum_{K\in{\mathcal{K}}}\mathbb{P}(A_K)\mathbb{P}(B_K)\mathbb{E}[\hat\xi_{[K]}]\nonumber\\
& = & \mathbb{E}[\hat X_{(n)},D_N]
- \sum_{K\in{\mathcal{K}}}\mathbb{P}(A_K)\mathbb{P}(B_K)\mathbb{E}[\hat\xi_{[K]}]\label{General}
\end{eqnarray}
Therefore, with $p_i=\mathbb{P}[X_i\in I]$,
$\mathbb{E}[X_i^{(n)}]=\mathbb{E}[Y_i^{(n)}]$ if and only if
\begin{eqnarray*}
\mathbb{E}[\hat X_i^{(n)},\hat X_i^{(n)}\leq b-a]
& = & \sum_{m=1}^n{n\choose m}p_i^mF_i^-(a)^{n-m}(1-\alpha_i^m)(b-a)\\
& = & (b-a)\left[\sum_{m=1}^n{n\choose m}p_i^mF_i^-(a)^{n-m}-\sum_{m=1}^n{n\choose m}p_i^m\alpha_i^mF_i^-(a)^{n-m}\right]\\
& = & (b-a)[F_i(b)^n-(F^-_i(a)+p_i\alpha_i)^n]
\end{eqnarray*}
that is
\begin{equation}
\left(\frac{F^-_i(a)+p_i\alpha_i}{F_i(b)}\right)^n(b-a) =
(b-a) - \mathbb{E}[\hat X_i^{(n)}|\hat X_i^{(n)}\leq b-a].
\label{EXnEYn}
\end{equation}
Returning to (\ref{General}), we find
\begin{eqnarray*}
\lefteqn{\mathbb{E}[X_{(n)}]-\mathbb{E}[Y_{(n)}]}\\
& = & \mathbb{E}[\hat X_{(n)},D_N]
- \sum_{K\in{\mathcal{K}}}\mathbb{P}(A_K)\mathbb{P}(B_K)\mathbb{E}[\hat\xi_{[K]}]\\
& = & \mathbb{E}[\hat X_{(n)},\hat X_{(n)}\leq b-a]
- \sum_{K\in{\mathcal{K}}}\prod_{k\in K}p_k\prod_{k\not\in K}F^-_k(a)\left(1-\prod_{k\in K}\alpha_k\right)(b-a)\\
& = & \mathbb{E}[\hat X_{(n)}|\hat X_{(n)}\leq b-a]\prod_{k=1}^nF_k(b)\\
& & - \left[\prod_{k=1}^nF_k(b)-\prod_{k=1}^nF^-_k(a)-\prod_{k=1}^n(F^-_k(a)+p_k\alpha_k)+\prod_{k=1}^nF^-_k(a)\right](b-a)\\
& = & \mathbb{E}[\hat X_{(n)}|\hat X_{(n)}\leq b-a]\prod_{k=1}^nF_k(b)
- \left[\prod_{k=1}^nF_k(b)-\prod_{k=1}^n(F^-_k(a)+p_k\alpha_k)\right](b-a)\\
& = & \prod_{k=1}^nF_k(b)\left\{\mathbb{E}[\hat X_{(n)}|\hat X_{(n)}\leq b-a]
- (b-a)+(b-a)\prod_{k=1}^n\frac{F^-_k(a)+p_k\alpha_k}{F_k(b)}\right\}.
\end{eqnarray*}
Now, by a simple application of H\"older's inequality, we get
\begin{eqnarray*}
\left((b-a) - \mathbb{E}[\hat X_{(n)}|\hat X_{(n)}\leq b-a]\right)^n
& = & \left(\int_0^{b-a}\prod_{k=1}^n(1-G_k(z))dz\right)^n\\
& \leq & \prod_{k=1}^n\int_0^{b-a}(1-G_k(z))^ndz\\
& = & \prod_{k=1}^n\left((b-a) - \mathbb{E}[\hat X_k^{(n)}|\hat X_k^{(n)}\leq b-a]\right)\\
& = & (b-a)^n\prod_{k=1}^n\left(\frac{F^-_k(a)+p_k\alpha_k}{F_k(b)}\right)^n\\
\end{eqnarray*}
where $G_k$ is the conditional distribution function of $(b-a)-\hat X_k$ given
$\{\hat X_k\leq b-a\}$ and we have used (\ref{EXnEYn}). It immediately follows (recall that
$(b-a) - \mathbb{E}[\hat X_k^{(n)}|\hat X_k^{(n)}\leq b-a]\geq0$) that
$$(b-a) - \mathbb{E}[\hat X_{(n)}|\hat X_{(n)}\leq b-a] \leq
(b-a)\prod_{k=1}^n\left(\frac{F^-_k(a)+p_k\alpha_k}{F_k(b)}\right)$$
which completes the proof.
\end{proof}

A consequence of this proposition is an improved lower bound in
the case of bounded random variables.

\begin{corollary}
If $X_1,\ldots,X_n$ are independent bounded random variables with
the property that $\mathbb{E}[X_i^{(n)}]=M_i$, $i=1,2,...,n$ ,
then
$$\mathbb{E}[X_{(n)}]\geq b-\prod_{i=1}^n\sqrt[n]{b-M_i}\geq \bar M,$$
where $b$ is a common upper bound to all $X_i$'s.\label{BoundCase}
\end{corollary}

\begin{proof}
Applying Proposition \ref{down} to the interval $I=[0,b]$ and the
random variables $Y_i$ such that $\bds\mathbb{P}[Y_i=0] =
1-\mathbb{P}[Y_i=b] = \sqrt[n]{1-\frac{M_i}b}\eds$, we immediately
get that
$$\mathbb{E}[X_{(n)}] \geq \mathbb{E}[Y_{(n)}] =
\left(1-\prod_{i=1}^n\sqrt[n]{1-\frac{M_i}b}\right)b =
b-\prod_{i=1}^n\sqrt[n]{b-M_i} \geq \bar M.$$ The second
inequality is a direct consequence of the inequality for
arithmetic and geometric means.
\end{proof}

\subsection*{Acknowledgement}
This research was supported by the Australian Research Council.
The authors would also like to thank the anonymous referees for
their most valuable comments, in  particular about the existing
literature.

\end{document}